\documentclass[fleqn]{amsart}
\usepackage[indentafter]{titlesec}
\titleformat{name=\section}{}{\thetitle.}{0.8em}{{\bf \color{bluegray}}}
\titleformat{name=\subsection}[runin]{}{\thetitle.}{0.8em}{{\bf \color{bluegray}}}[.]
\titleformat{name=\subsubsection}[runin]{}{\thetitle.}{0.5em}{\itshape}[.]
\titleformat{name=\paragraph,numberless}[runin]{}{}{0em}{}[.]
\titlespacing{\paragraph}{0em}{0em}{0.5em}
\titleformat{name=\subparagraph,numberless}[runin]{}{}{0em}{}[.]
\titlespacing{\subparagraph}{0em}{0em}{0.5em}
\def\tE{\widetilde{E}}
\usepackage[latin1]{inputenc}
\usepackage{dsfont}
\usepackage{amsmath}
\usepackage{amssymb}
\usepackage{graphicx}
\usepackage{amssymb}
\def\hB{\hat{B}}
\usepackage{amsfonts}

\usepackage{soul}

\usepackage{palatino}
\def\hk{\hat{k}}
\def\cR{\mathcal R} 
\usepackage{color}
\usepackage{yfonts}

\usepackage{paralist}

\usepackage{stmaryrd}

\usepackage{amsxtra}

\def\t$N{\tilde{N}}

\def\eps{\varepsilon}

\def\a{          \alpha}

\def\C{          \mathbb C}

\def \g{{\gamma}}

\def \R{{\mathbb R}}

\def \Z{{\mathbb Z}}
\def \N{{\mathbb N}}
\def \t{{\theta}}

\newcommand{\T}{{\mathbb T}}
\newcommand{\prf}{{\begin{proof}}}
\newcommand{\epf}{{\bg \end{proof} \black }}

\newtheorem{theo}{\sc \bg Theorem}
\newtheorem{prop}{\sc \bg Proposition}
\newtheorem{lemm}{\sc \bg Lemma}

\newtheorem{coro}{\sc \bg Corollary}

 \newcommand{\tH}{\widetilde{H}}
 \newcommand{\tF}{\widetilde{F}}

\def\pa{\partial}
\def\bee{\begin{equation}}
\def\eee{\end{equation}}

\newcommand{\pdvr}[2]
{\dfrac{\partial^{#2} #1}{\partial \theta^{#2_1} \partial r^{#2_2}}}
\newcommand{\pdvrs}[2]
{\partial^{#2} #1 /\partial \theta^{#2_1} \partial r^{#2_2}}
\newcommand{\bxi}{{\\bar \xi}}

\def\bbH{h}
\def\om{\omega}

\def\bxi{\bar \xi}
\def\beta{\bar \eta}
\title{Lyapunov unstable elliptic equilibria}
\author[Bassam Fayad]{Bassam Fayad}
\usepackage{nccmath}
\usepackage{amsthm}
\definecolor{bluegray}{rgb}{0.1, 0.1, 0.6}%{0.4, 0.6, 0.8}
\usepackage{xcolor}

\def\bluegray{\color{bluegray}}
\def\black{\color{black}}

\usepackage{xcolor}
\def\bg{\color{bluegray}}

\def\III{I_3}
\usepackage[numbers]{natbib}

\begin{document}
\nocite{*}
\setcounter{tocdepth}{2}
 \maketitle 
 %\tableofcontents

\bibliographystyle{abnt-num}
%\bibliography{mixmaster.bib}  % <--- For testing, replace this by your own bib tile

\maketitle

{\color{bluegray} \begin{abstract} {\color{black} A new diffusion mechanism  from the neighborhood of elliptic equilibria for Hamiltonian flows in three or more degrees of freedom is introduced. We thus obtain explicit real entire Hamiltonians on $\R^{2d}$, $d\geq 4$, that have a Lyapunov unstable elliptic equilibrium with an arbitrary chosen frequency vector whose coordinates are not all of the same sign. For non-resonant frequency vectors, our examples all have divergent Birkhoff normal form at the equilibrium. 

On $\R^4$, we give explicit examples of real entire Hamiltonians having an equilibrium with an arbitrary chosen non-resonant frequency vector and a divergent  Birkhoff normal form.
}
\end{abstract}}

\color{bluegray}
\section*{ \bf  Introduction} 
\medskip 
\color{black}

 An equilibrium $(p,q)\in \R^{2d}$ of an autonomous Hamiltonian flow  is said to be  Lyapunov stable or topologically stable if all nearby orbits remain close to $0$ for all forward time. 

The topological stability of the equilibria of Hamiltonian flows is one of the oldest problems in mathematical physics. The important contributions to the understanding of this problem, dating back to the 18th century, form a fundamental part of the foundation and of the evolution of the theory of dynamical systems and celestial mechanics up to our days. 

The goal of this note is to give  examples of real analytic Hamiltonians that have a Lyapunov unstable non-resonant elliptic equilibrium.

\bigskip

A $C^2$ function $H:(\R^{2d},0)\to\R$ such that $DH(0)=0$ defines a Hamiltonian vector field $X_{H}(x,y)=(\pa_{y}H(x,y),-\pa_{x}H(x,y))$  whose flow $\phi^t_{H}$  preserves the origin. 
%, linearly stable if all orbits of the tangent flow are bounded, and spectrally stable if all eigenvalues of the tangent flow are pure imaginary. For example, consider the nonlinear pendulum of length l and mass m , as shown in Figure 1 and described by the Hamiltonian 

Naturally, to study the stability of the equilibrium at the origin,  one has first to investigate the stability of the linearized system %given by the Hamiltonian flow of the Hessian of $H$
 at the origin. By symplectic symmetry, the eigenvalues of the linearized system come by pairs $\pm \lambda$, $\lambda \in \C$. 
It follows that if the linearized system has an eigenvalue with a non zero real part, it also has an eigenvalue with positive  real part and this implies instability of the origin for the linearized system as well as for the non-linear flow. 

When all the eigenvalues of the linearized system are on the imaginary axis the stability question is more intricate. In the non-degenerate case where the eigenvalues are simple, we say that the origin is an {\it elliptic} equilibrium. The linear system is then symplectically conjugated to a direct product of planar rotations. The arguments of the eigenvalues are called the frequencies of the equilibrium since they correspond to angles of rotation of the linearized system.  We  
focus our attention on Hamiltonians $H:(\R^{2d},0)\to\R$ of the form \begin{align} \label{star}
 \tag{$*$} H(x,y)&=H_\om(x,y)+{\mathcal O}^3(x,y),\\  \nonumber H_\om(x,y)&=\sum_{j=1}^3 \om_j I_j, \quad I_j=\frac12(x_j^2+y_j^2). \end{align}
where $\omega\in\R^d$ has rationally independent coordinates. The elliptic equilibrium at the origin of the flow of $X_H$ is then said to be {\it non-resonant}.

The phenomenon of averaging out of the non-integrable part of the nonlinearity effects at a non-resonant frequency is responsible for the long time effective stability around the equilibrium : the points near the equilibrium remain in its neighborhood during a time that is greater than any negative power of their distance to the equilibrium. This can be formally studied and proved using the Birkhoff Normal 
Forms (BNF) at the equilibrium, that introduce action-angle coordinates in which the system is integrable up to arbitrary high degree in its Taylor series (see Section \ref{sec.BNF} for some reminders about the BNF, and \cite{Bi} or \cite{SM}, for example, for more details). Moreover, it was proven in \cite{MG,BFN} that a typical elliptic fixed point is doubly exponentially stable in the sense that 
a neighboring  point of  the equilibrium  remains close to it
for an interval of time which is doubly exponentially large with respect to some power of the inverse
of the distance to the equilibrium point.

In addition to the long time effective stability of non-resonant equilibria, KAM theory (after Kolmogorov Arnold and Moser), asserts that a non-resonant  elliptic fixed point is in general accumulated by quasi-periodic invariant Lagrangian tori whose relative measurable density tends to one in small neighborhoods of the fixed point. This can be viewed as stability in a probabilistic sense, and is usually coined {\it KAM stability}. In classical KAM theory, KAM stability is established when the BNF has a non-degenerate Hessian. 
 Further development of the theory allowed to relax the non degeneracy condition and \cite{EFK_point} proved KAM-stability of a non-resonant elliptic fixed point under the non-degeneracy condition of the BNF (see Section \ref{sec.BNF}).

Despite the long time effective stability, and despite the genericity of KAM-stability, Arnold conjectured that apart from two cases, the case of a sign-definite quadratic part of the Hamiltonian, and generically for $d=2$, an elliptic equilibrium point of a generic real analytic Hamiltonian system is Lyapunov unstable \cite[Section 1.8]{arnold}.

%\noindent{\sc Conjecture}[Arnold] {\it An elliptic equilibrium point of a generic analytic Hamiltonian system is Lyapounov unstable, provided $n \geq 3$ and the quadratic part of the Hamiltonian function at the equilibrium point is not sign-definite.} 

Although a rich literature in the direction of proving this conjecture exist in the $C^\infty$ smoothness (we mention \cite{KMV} below, but to give a list of contributions would exceed the scope of this introduction), the conjecture is still wide open in the real analytic category. For instance, not a single example of real analytic Hamiltonians was known that  has an unstable non-resonant elliptic equilibrium. 
The main goal of this work is to give the first examples of real analytic Hamiltonians having an unstable non-resonant elliptic equilibrium, with an arbitrary frequency vector for $d\geq 4$. 

Inspired by the construction of these examples, we will obtain the first explicit examples of real entire Hamiltonians having an elliptic equilibrium at the origin and a divergent BNF, and this for {\it any} degree of freedom and for any non-resonant frequency vector.

\color{bluegray}
\section{ \bf Unstable equilibria and divergent Birkhoff normal forms. General statements} 
\medskip 
\color{black}

\color{bluegray} 
\subsection{Lyapunov unstable equilibria} 
\color{black}

We start with  the existence of real entire Hamiltonians with unstable non-resonant equilibria. 
 
 \medskip 
 
\noindent {\color{bluegray} \sc Theorem {\bg A}.  --  } {\it  There exists a non-resonant $\om \in \R^3$ and a real entire Hamiltonian $H : \R^6 \to \R$, such that the origin is a Lyapunov unstable elliptic equilibrium with frequency $\om$ of the Hamiltonian flow $\Phi^t_H$ of $H$.

For any $\om \in \R^{d}, d \geq 4$, such that not all its coordinates are of the same sign, there exists a real entire Hamiltonian $H: \R^{2d} \to \R$ such that the origin is a Lyapunov unstable elliptic equilibrium with frequency $\om$ of the Hamiltonian flow $\Phi^t_H$ of $H$.

Moreover, in all our examples, for non-resonant frequencies $\om$, the Birkhoff normal form at the origin is divergent.

Finally, for non-resonant frequencies $\om$, it is possible to choose the Hamiltonians $H$ such that the origin is KAM stable. }

\medskip 

Detailed statements with an explicit definition of the Hamiltonians that prove Theorem {\bg A}, will be given in Section \ref{sec.statements}.

Note that we do not obtain the existence of an orbit that accumulates on the origin.  Based on a different diffusion mechanism, \cite{FMS} gives examples of smooth symplectic diffeomorphisms of $\R^6$ having a non-resonant elliptic fixed  point that attracts an orbit. 

Note also that  the question of Lyapunov instability in two degrees of freedom remains open. The question remains open also in the case of three degrees of freedom and Diophantine frequency vectors.

As explained earlier,  Lyapunov instability of an elliptic fixed point of frequency vector $\om$ is only possible when not all  the coordinates of $\om$ are of the same sign.

\bg \subsection{\bf \bg The case of quasi-periodic tori} \black The same constructions that yield Theorem {\bg A} can be carried out on $\R^d \times \T^d$ to get examples, starting from $d=3$, of real analytic Hamiltonians with an invariant quasi-periodic torus $\{0\}\times \T^d$ that is Lyapunov unstable. Moreover, in that case, the condition that the coordinates of the frequency vector of the quasi-periodic torus are not all of the same sign is not anymore required. We will explain this in Section \ref{sec.torus} after the explicit form of the Hamiltonians with Lyapunov unstable equilibria is given. We do not pursue the constructions on $\R^d \times \T^d$ in detail in this paper, because the work  \cite{FF} provides many examples of real analytic Hamiltonians with  invariant quasi-periodic tori $\{0\}\times \T^d$ that are Lyapunov unstable. The construction method of \cite{FF} is different from the one introduced here. The constructions of \cite{FF}  are limits of successive conjugacies of integrable Hamiltonians and as such, in contrast with the examples given here, they have a convergent Birkhoff normal form at the invariant quasi-periodic torus. 
 The method of \cite{FF} is for the moment inapplicable to equilibrium points and the question of having unstable elliptic equilibrium with a convergent BNF is still completely open in the real analytic setting.

\color{bluegray}
\subsection{Generic divergence of Birkhoff normal forms for all $d\geq 2$} \color{black}   Inspired by the constructions of Theorem {\bg A}, it is possible to obtain explicit examples of real entire Hamiltonians having an elliptic equilibrium at the origin and a divergent BNF for {\it any} degree of freedom including $d=2$, and for any non-resonant frequency vector, including vectors whose coordinates are all of the same sign.  The difference with the examples of Theorem {\bg A} is that the divergence of the BNF in the case $d=2$ and the case all the coordinates of the frequency vector are of the same sign, do not give much informations about the asymptotic dynamics in the neighborhood of the origin. 

%However, if one only seeks the divergence of the Birkhoff normal form, then it is possible to modify the construction to have examples for any non-resonant frequency in $\R^d$, $d\geq 4$. 

\medskip 
 
\noindent {\color{bluegray} \sc Theorem {\bg B}.  --  } {\it 
For any non-resonant $\om \in \R^{d}, d \geq 2$, 
there exists a real entire Hamiltonian $H: \R^{2d} \to \R$ such that the origin is an elliptic equilibrium with frequency $\om$ of the Hamiltonian flow $\Phi^t_H$, and such that the Birkhoff normal form at the origin is divergent.}

\medskip

The proof of Theorem {\bg B} is inspired from the proof of the divergence of the BNF in the examples of Theorem {\bg A}. However, to include the  two degrees of freedom case requires a substantial difference  on which we will  comment once the examples are explicitly  stated in Section \ref{sec.newBNF}. In the end of Section \ref{sec.newBNF}, we explain the slight modification required to prove Theorem {\bg B} in the case of non-resonant vectors whose coordinates are all of the same sign.

 %, just after the statement of Theorem \ref{theo.4} where the explicit construction of the unstable Hamiltonians in $\R^{2d}$, $d\geq 4$ is given.

Note that, due to the result of Perez-Marco of \cite{PMannals}, the existence for any non-resonant $\om \in \R^d$ of just one example of a real analytic Hamiltonian with divergent BNF, implies that divergence of the BNF is typical for this frequency. Denote by $\mathcal H_\om$, the set of  analytic Hamiltonians having an elliptic fixed point of frequency $\om$ at the origin. As a consequence of Theorem {\bg B} and of \cite[Theorem 1]{PMannals} we get 

\bigskip 

\noindent {\color{bluegray} \sc Corollary.  --  } {\it 
For any non-resonant $\om \in \R^{d}, d \geq 2$, the generic Hamiltonian in $\mathcal H_\om$ has a divergent BNF at the origin.  

More precisely, all Hamiltonians in any complex (resp. real ) affine finite-dimensional subspace $V$ of $\mathcal H_\om$ have a divergent BNF except  for  an exceptional pluripolar set.}

\bigskip

This answers, for all frequency vectors the question of Eliasson on the typical behavior of the BNF (see for example \cite{E1,E2,EFK} and the discussion around this question in \cite{PMannals}). What was known up to recently, was the generic divergence of the normalization, proved by Siegel in 1954 \cite{siegel} in $\mathcal H_\om$ for any fixed $\om$. Examples of analytic Hamiltonians with non-resonant elliptic fixed points and divergent BNF were constructed by 
 Gong \cite{gong} on $\R^{2d}$ for arbitrary $d\geq 2$, but only for some class of Liouville frequency vectors.

The  generic divergence of the BNF was recently obtained by Krikorian for symplectomorphisms 
of the plane with an elliptic fixed point at the origin \cite{kriko}. The method of Krikorian is completely different from ours and does not rely on the dichotomy proved by Perez-Marco. He has an indirect proof that gives a more refined result than the generic divergence of the BNF. Indeed, he proves that the convergence of the BNF, combined with torsion (a generic condition), implies the existence of a larger measure set of invariant curves in small neighborhoods of the origin than what  actually  holds for a generic symplectomorphism. 

%In contrast, our proof does not yield any result in dimension lower than $6$ for flows. In dimension $6$, the generic divergence of the BNF will hold generically in $\mathcal H_\om$, but only for super Liouville frequencies $\om$ such as the ones we will use in our construction of Lyapunov unstable fixed points (see condition \eqref{liouville} of Section \ref{sec.statements}, although a less restrictive Liouville condition is sufficient for the divergence of the BNF, as seen in the last line of the proof of Theorem \ref{theo.3}). 

\bigskip

\color{bluegray}
\subsection{About the diffusion mechanism that will be used} \color{black} 
In the $C^\infty$ category, examples of unstable elliptic equilibria can be obtained {\it via} the successive conjugation method,  the Anosov-Katok method. They can be obtained in two degrees of freedom or for $\R^2$ symplectomorphisms, provided the frequency at the elliptic equilibrium is not Diophantine (\cite{AK,FS,FSpoint}).  In three or more degrees of freedom, smooth examples with Diophantine frequencies can be obtained through a more sophisticated version of the successive conjugation method (see \cite{EFK,FSpoint}).   The Anosov-Katok examples are infinitely tangent to the rotation of frequency $\om$ at the fixed point and as such are very different in nature from our construction. In particular, KAM stability is in general excluded in these constructions. 

Again in the $C^\infty$ class but in the non-degenerate case, R. Douady gave examples in  \cite{douady} of Lyapunov unstable elliptic points for  symplectic diffeomorphisms on $\R^{2d}$ for any $d\geq 2$. Douady's examples can have any chosen   Birkhoff Normal Form at the origin provided its Hessian at the fixed point is non-degenerate. Douady's examples are modeled on the Arnold diffusion mechanism through chains of heteroclinic intersections between lower dimensional partially hyperbolic invariant tori that accumulate toward the origin. The construction consists of a countable number of compactly supported perturbations of a completely integrable flow, and as such was carried out only in the $C^\infty$ category. 

In \cite{KMV}, the authors admit Mather's proof of Arnold diffusion for a cusp residual set of nearly integrable convex
Hamiltonian systems in 2.5 degrees of freedom, and deduce from it that generically, a convex resonant totally elliptic
point of a symplectic map in 4 dimensions is Lyapunov unstable, and in fact has orbits that converge to the fixed point.

A third diffusion mechanism, closely related to Arnold diffusion mechanism, is Herman's synchronized diffusion, and is due to Herman, Marco and Sauzin \cite{MS}. It is based on the following coupling of two twist maps of the annulus (the second one being integrable with linear twist): at exactly one point $p$ of a well chosen periodic orbit of period $q$ on the first twist map, the coupling consists of pushing the orbits in the second annulus up on some fixed vertical $\Delta$ by an amount that sends an invariant curve whose rotation number is a multiple of $1/q$ to another one having the same property. The dynamics of the coupled maps on the line $\{p\}\times \Delta$ will thus drift at a linear speed. %The adaptation of Herman's synchronized diffusion to the instability of elliptic equilibria was done in \cite{fayad-sauzin}. 

 The diffusion mechanism that underlies our constructions is inspired by all these three mechanisms described above but is quite different from each. In 3 degrees of freedom, we start with a product of rotators of frequencies $\om_1,\om_2,\om_3$, where $\om_1 \om_2<0$ 
and then  perturb this integrable Hamiltonian by adding a monomial of the $4$ coordinates $(x_1,y_1,x_2,y_2)$ that has a diffusive multi-saddle at the origin. The perturbation almost commutes with the rotators, provided  $\bar \om=(\om_1,\om_2)$ is very well approached by resonant vectors. The perturbed system has then an orbit that starts very close to the origin and that diffuses in the first four coordinates as is the case for the resonant system.  
We use the third action, $I_3=x_3^2+y_3^2$, that is invariant by the whole flow, as a coupling parameter.  
%If only one coupling term is added in this way to the original Hamiltonian, there appears an orbit that diffuse from a small (but not arbitrarily small) neighborhood of the origin towards infinity. %(in projection on the $(x_1,y_1,x_2,y_2)$ space, since the third action is invariant). 

To get diffusion from arbitrary small neighborhoods of the origin, one has to add successive couplings that commute with increasingly better resonant approximations of $\bar \om$. The use of the third action as a coupling parameter, allows to isolate the effect of each successive coupling from the other ones. Indeed, to isolate the effect of each individual coupling from all the successive couplings is easy because these terms can be chosen to be extremely small compared to it. On the other hand, if we look at adequately small values of the third action, the effect of the prior coupling terms is tamed out due to Birkhoff averaging (we refer to Section \ref{sec.description} for a more precise description of the diffusion mechanism).   

In the case of $4$ degrees of freedom (or more) we can take the frequency vector of the equilibrium to be arbitrary, provided all the coordinates are not of the same sign, assuming for definiteness $\om_1\om_2<0$. The idea is that if $\omega_1$ is replaced by $\om_1+I_4$ then 
the vector $(\om_1+I_{4,n},\om_2)$ will be resonant, for a sequence $I_{4,n}\to 0$, which allows to adopt the  three degrees of freedom  diffusion strategy.

%Following an idea introduced in \cite{EFK} (see also \cite{FSpoint}) we can use the action of the fourth degree of freedom as a parameter, that merely changes the frequency $\om_1$ into $\om_1+I_4$. For a sequence $I_{4,n} \to 0$ 

%the Hamiltonians we consider will be a coupling of a Hamiltonian of the first three degrees of freedom and a rotator of frequency $\om_4$ in the fourth degree of freedom, and with a coupling that only depends on the fourth action variable $I_4=\frac12(x_4^2+y_4^2)$.  Hence the action variable $I_4$ will be constant along the motion. 
%To obtain the diffusion in the first three degrees of freedom, we first consider the three rotators  of frequencies $,\om_2,\om_3$. Regardless of the values of $\om_1$ and $\om_2$ the vector   $\om_1+I_4,\om_2$ will go through rational resonances as $I_4 \to 0$, which opens the way for the application of the diffusion mechanism in $3$ degrees of freedom that we discussed above, in an even simpler context in fact, since the Liouville condition on $(\om_1,\om_2)$ is in this case an exact resonance condition. }

\bluegray
\section{\color{bluegray} Birkhoff Normal Forms}
\black \label{sec.BNF}
%Notations and reminders} 

\medskip 

%Let $(x,y)=(x_1,\ldots,x_d,y_1,\ldots,y_d)$ be symplectic coordinates defined on $\R^{2d}$.
  
%The Hamiltonian system associated to $H$ is given by the vector field $X_H=(\pa_{y}H, -\pa_{x}H)$, that is$$  \quad\left\{\begin{array}{l}\dot x=\pa_{y}H(x,y)\\ \dot y=-\pa_{x}H(x,y).\end{array}\right.$$

%\begin{equation*} H_\om(\xi,\eta)=\sum_{j=1}^3 \xi_j \eta_j.\end{equation*}
%and

%Given a sequence $(k_n,l_n)\in \N^*\times \N^*$, we will consider the following real entire Hamiltonians  
%\begin{equation} \label{eqFn} F_n(\xi_1,\xi_2,\eta_1,\eta_2)=\xi_1^{k_n}\xi_2^{l_n} +\eta_1^{k_n}\eta_2^{l_n}.\end{equation}

%For $d \geq 1$ and $H\in C^2(\C^{2d},\C)$  we denote by $\Phi_{H}^t$ the Hamiltonian flow of $H$ 

For $H$ as in \eqref{star}, $\om$ non-resonant, for all $N\geq 1$, there exists an exact symplectic transformation $\Phi_N={\rm Id} +O^2(x,y)$, and a polynomial 
$B_N$ of degree $N$ in the variables $I_1,\ldots,I_d$, such that 
$$H\circ \Phi_{N}(x,y)=B_{N}(I)+O_{2N+1}(x,y).$$
We say that  $f \in O^l(x,y)$ when $\partial_z f(0)=0$ for any multi-index $z$ on the $x_i$ and $y_i$ of size less or equal to $l$.  

There also exists a formal   exact symplectic transformation $\Phi_\infty={\rm Id} +O^2(x,y)$, where $O^2(x,y)$ is formal power series in $x$ and $y$ with no affine terms such that  
$$H\circ \Phi_{\infty}(x,y)=B_{\infty}(I)$$
where $B_{\infty}$ is a uniquely defined formal power series of the action variables $I_j$, called the Birkhoff Normal Form (BNF) at the origin.

For more details on the Birkhoff Normal Form
at a Diophantine, and more generally at any non-resonant elliptic equilibrium, one can consult for example \cite{SM}.

\bigskip

\noindent {\color{bluegray} \indent Divergent  Birkhoff Normal Forms.} When the radius of convergence of the formal power series $B_\infty(\cdot)$ is $0$, we say that the BNF diverges.

%Assume that $H(z,w)$ is a formal Hamiltonian in $\bC[[z,w]]$ of the form
%$$H(z,w)=\<\om_0,zw\>+\OO^3(z,w)$$
%with a vector $\om_0\in\R^d$ which is rationally independent.
 %It is a classical result that there exist a unique $N\in\bR[[r]]$, the Birkhoff Normal Form of $H$, and  a (formally) exact 
% mapping of the form
%$$(*)\quad Z(z,w)-(z,w)\in \C[[z,w]]\cap\OO^2(z,w)$$ 
 %such that 
%$$H\circ Z(z,w)=N(zw)= \<\om_0,zw\>+\OO^3(zw).$$

\bigskip 

\noindent {\color{bluegray} \indent Non-degenerate Birkhoff Normal Forms. }  Following \cite{russmann}, we say that $B_\infty$ is {\it R\"ussmann non-degenerate} or simply {\it non-degenerate} if there does not exist any vector  $\g$ such that for every   $I$ in some neighborhood of $0$
$$\langle \nabla B_\infty(I),\g\rangle=0.$$

In \cite{EFK} the following was proven 
\begin{theo}[\cite{EFK}]   -- \  \label{mB}  Let $H:(\R^{2d},0)\to \R$ be a real  analytic function of the form $(*)$
and assume that  $\omega$ is non-resonant. 
If  the BNF of $H$ at the origin  is non-degenerate, then  in any neighborhood of
$0\in\R^{2d}$ the set of real analytic KAM-tori for $X_H$ is of
positive Lebesgue measure and density one at $0$. 
\end{theo}

 Real analytic KAM-tori are invariant Lagrangian  tori on  which the flow  generated by $H$ is real analytically conjugated to a minimal translation flow on the torus $\R^d/\Z^d$. 

\bigskip 

\bigskip 

\noindent {\color{bluegray} \indent Calculating the BNF : Resonant and non-resonant terms.}

To simplify the computations, we prefer to use
the complex variables 
\begin{equation*} \xi_j=\frac{1}{\sqrt{2}} (x_j+iy_j), \quad \eta_j= \frac{1}{\sqrt{2}} (x_j-iy_j)\end{equation*}
Notice in particular that the
actions become $I_j=\xi_j \eta_j$, making it very simple to detect the monomials $\xi_{u_1}\ldots \xi_{u_k}\eta_{v_1}\ldots \eta_{v_{k'}}$ that only depend on the actions, since these are exactly the monomials for which $k=k'$ and $\{u_1,\ldots,u_k\}=\{v_1,\ldots,v_k\}$. 

Note also that in these variables $H_\om$ as in \eqref{star} reads as $\sum \om_j \xi_j \eta_j$. We easily verify that, in these variables,
the Poisson bracket is given by 
\begin{equation*} \{F,G\} = i \sum_{j} \frac{\partial F}{\partial \xi_j}\frac{\partial G}{\partial \eta_j}-\frac{\partial F}{\partial \eta_j}\frac{\partial G}{\partial \xi_j},\end{equation*} 
while the Hamiltonian equations are given by 
\begin{equation*}   \label{Hcomplex}   \begin{cases}\dot{\xi}_j = -i \partial_{\eta_j} H(\xi,\eta) \\ \dot{\eta}_j = i \partial_{\xi_j} H(\xi,\eta)  \end{cases}\end{equation*} 
We will say that a function $F$ defined in the variables $\xi$ and $\eta$ is real when $F(\xi,\bxi)$ is real, which means that in the original variables $(x,y)$, $F$ is real valued. 

Finally observe that for a Hamiltonian as in \eqref{star}, and since $\om$ is non-resonant, it is easy to eliminate by conjugacy all the terms $\xi_{u_1}\ldots \xi_{u_k}\eta_{v_1}\ldots \eta_{v_{k'}}$ that do not depend only on the actions. Indeed, if we take 
$$\chi=i \frac{1}{\om_{u_1}+\ldots+\om_{u_k}-\om_{v_1}-\ldots-\om_{v_{k'}} } \xi_{u_1}\ldots \xi_{u_k}\eta_{v_1}\ldots \eta_{v_{k'}}$$
we get for the time one map $\Phi^1_{\chi}$ of the Hamiltonian flow of $\chi$, also called the Lie transform associated to $\chi$,  
$$H \circ \Phi^1_{\chi}= H+\{H_\om,\chi\}+\{H-H_\om,\chi\}+\frac{1}{2!}\{\{H,\chi\},\chi\}+\ldots  $$
and observe that $\{H_\om,\chi\}=-\xi_{u_1}\ldots \xi_{u_k}\eta_{v_1}\ldots \eta_{v_{k'}}$, while the other terms that appear due to the composition by $\Phi^1_{\chi}$ are of higher degree than $k+k'$. This is why in the case of a non-resonant vector $\omega$, we call the action dependent monomials {\it resonant} and the others non-resonant. The reduction to the BNF is done progressively by eliminating  non-resonant monomials of higher and higher degree. In the case of a resonant frequency vector $\omega$ one cannot formally conjugate the Hamiltonian to an action dependent formal power series, because monomial terms $\xi_{u_1}\ldots \xi_{u_k}\eta_{v_1}\ldots \eta_{v_{k'}}$  that do not only depend on the actions may be {\it resonant with $\omega$}, namely when ${\om_{u_1}+\ldots+\om_{u_k}-\om_{v_1}-\ldots-\om_{v_{k'}} }=0$. Such resonant frequencies and resonant monomials that do not only depend on the actions will be instrumental in our constructions.
 
\bluegray
\section{\color{bluegray}  Explicit Hamiltonians with Lyapunov unstable elliptic equilibria} \label{sec.statements}
\black
In this section, we give the explicit constructions that yield Theorem {\bg A}. 

%Our examples in three degrees of freedom require very strong almost resonances on the vector $(\om_1,\om_2)$, where we assumed that $\om_1 \om_2<0$. 

Starting from $4$ degrees of freedom, it is possible to give 
examples with arbitrary frequency vectors, in particular Diophantine. Recall that $\om$ is said to be Diophantine if  there exists $\gamma,\tau>0$ such that $|\langle k,\omega\rangle|\geq \gamma |k|^{-\tau}$, for all $k\in\Z^d- \{0\}$, with $\langle \cdot\rangle$ being the canonical scalar product and $|\cdot|$ its associated norm. 

In his ICM talk of 1998 \cite{herman_icm}, Herman conjectured that a real analytic elliptic equilibrium with a Diophantine frequency vector must be accumulated by a set of positive measure of KAM tori. This conjecture is still open.  However, our examples can be chosen such that the Birkhoff Normal Form is non-degenerate, which implies KAM-stability as established in \cite{EFK_point} (see Theorem \ref{mB} above).  

In all the sequel,  we denote $|\cdot|$ the Euclidean norm on $\R^{2d}$, indifferently on the value of $d$ that will be clear from the context.  We also denote indifferently $B_r$ the  Euclidean  ball of radius $r$ in $\R^{2d}$ for any value of $d$. For $k \in \N$, we denote by  $\|H\|_{C^k(B_{R})}$ the $C^k$ norm of $H$ on the ball $B_R$.

\bluegray \subsection{Lyapunov unstable elliptic equilibrium in three degress of freedom} \black \label{sec31}

\medskip

%For a frequency vector $\om \in \R^3$ we define the Hamiltonian on $\R^6$ 
%\begin{equation*} H_\om(x,y)=\sum_{j=1}^3 \om_j I_j, \quad I_j=\frac12(x_j^2+y_j^2).\end{equation*}

We suppose $\om\in \R^3$ is such that there exists a sequence $  \{(k_n,l_n)\} \in \N^*\times \N^*$ satisfying\footnote{The requirement of double exponential approximations is not uncommon in instability results in real analytic and holomorphic dynamics as is the case for example in \cite{PM}.} 
\begin{equation} \label{liouville} \tag{${\mathcal L}$} 0<|k_n \om_1+l_n\om_2|<e^{-e^{n^4(k_n+l_n)}}.\end{equation}
The set of vectors satisfying \eqref{liouville} is clearly a $G_\delta$-dense set, since resonant vectors form a dense set in $\R^2$. Up to extracting we can also assume that 
\begin{equation}\tag{${\mathcal NR}$} k_{n} \geq \max_{0<k+l\leq k_{n-1}+l_{n-1}} e^{\frac{1}{|k \om_1+l\om_2|}}.\end{equation}

For $n \in \N$ we define on $\R^4$ the following real polynomial Hamiltonians 
\begin{equation} \label{def.Fn} F_n(x_1,x_2,y_1,y_2)=\xi_1^{k_n}\xi_2^{l_n} +\eta_1^{k_n}\eta_2^{l_n}, \end{equation} %\quad \xi_j=\frac{1}{\sqrt{2}} (x_j+iy_j), \eta_j=\frac{1}{\sqrt{2}} (x_j-iy_j).\end{equation}

 We finally define a real entire Hamiltonian on $\R^6$ 
\begin{equation*}H(x,y)=H_\om(x,y)+\sum_{n\in \N} e^{-n(k_n+l_n)} I_3 F_n(x_1,x_2,y_1,y_2)\end{equation*}

\bigskip

\begin{theo} -- \ \label{theo.3} The origin is a Lyapunov unstable equilibrium of the Hamiltonian flow $\Phi^t_H$ of $H$. More precisely, for every $n \geq 1$, there exists $z_n \in \R^6$, such that $|z_n|\leq \frac{1}{n}$, and $\tau_n \geq 0$ such that $|\Phi^{\tau_n}_H(z_n)|\geq n$. 

Moreover, the Birkhoff normal form of $H$ at the origin is divergent.

\end{theo}

We can modify the definitions of the Hamiltonians $H_\om$ and $H$  on $\R^6$ as follows 
\begin{align*} \tH_\om(x,y)&= (\om_1+I_3^3) I_1+(\om_2+I_3^4) I_2+\om_3I_3,\\
\tH(x,y)&=\tH_\om(x,y)+\sum_{n\in \N} e^{-n(k_n+l_n)} I_3 F_n(x_1,x_2,y_1,y_2).\end{align*}

We can assume that $k_0+l_0>10$, hence $\tH_\om$ gives the BNF of $\tH$ at the origin up to order $5$ in the action variables. But $\nabla \tH_\om(I)= (\om_1+I_3^3, \om_2+I_3^4,\om_3+ 3I_3^2I_1+4I_3^3I_2)$ is clearly non-degenerate, and this implies that the BNF of $\tH$ is non-degenerate. 
We then have the following 

\bigskip 

\begin{theo} -- \ \label{theo.33} The origin is a Lyapunov unstable equilibrium of the Hamiltonian flow $\Phi^t_{\tH}$ of $\tH$. Moreover, the Birkhoff normal form of $\tH$ at the origin is non-degenerate, hence the equilibrium is KAM-stable. %Also,  the Birkhoff normal form of $H$ at the origin is divergent.

\end{theo}

\bluegray
\subsection{\color{bluegray} Lyapunov unstable elliptic equilibrium in four degrees of freedom} \label{sec32}
\black

In $4$ degrees of freedom (or more), our method yields unstable elliptic equilibria for any frequency vector, provided its coordinates are not all of the same sign. Suppose for instance that $\om=(\om_1,\ldots,\om_4)$ is such that $\om_1 \om_2<0$. 

We assume  $(\om_1,\om_2)$  non-resonant (the resonant case follows from Corollary  \ref{rem.resonant} below). By Dirichlet principle, there exists a sequence $(k_n,l_n)\in \N^*\times \N^*$  such that 
\begin{equation*} |k_n \om_1+ l_n\om_2| <\frac{1}{k_n^2}. \end{equation*}
WLOG, we assume that $k_n \om_1+ l_n\om_2<0$. Then, for  $I_{4,n}=  -(k_n \om_1+ l_n\om_2)/k_n \in  (0,\frac{1}{k_n^3})$, it holds  
that 
\medskip 

\begin{equation} \label{rational} \tag{${\mathcal R}$} k_n (\om_1+I_{4,n})+ l_n\om_2 =0. \end{equation}

Since $(\om_1,\om_2)$ is non-resonant,  we can, up to extracting, additionally ask that for all $(k,l)\in \N^2\setminus \{0,0\}$ such that $k+l\leq k_{n-1}+l_{n-1}$, we have $k(\om_1+I_{4,n})+l \om_2 \neq 0$ and
\begin{equation} \label{lower} \tag{${\mathcal NR}$} k_{n}\geq  \max_{0<k+l\leq k_{n-1}+l_{n-1}}  e^{\frac{1}{|k (\om_1+I_{4,n})+l\om_2|}}, \quad k_n\geq e^{e^{e^{n^4(k_{n-1}+l_{n-1})}}}.\end{equation} 

We define the following real entire Hamiltonians on $\R^8$ 
\begin{align*} H_\om(x,y)&=  (\om_1+I_4)I_1+ \sum_{j=2}^4 \om_j I_j, \\%\quad I_j=\frac12(x_j^2+y_j^2), \\
H(x,y)&=H_\om(x,y)+\sum_{n\in \N} e^{-n(k_n+l_n)} I_3 F_n(x_1,x_2,y_1,y_2)\end{align*}

\begin{theo} -- \ \label{theo.4} The origin is a Lyapunov unstable equilibrium for the Hamiltonian flow of $H$.  More precisely,  for every $n \geq 1$,  there exists $z_n \in \R^8$, such that $|z_n|\leq \frac{1}{n}$, and $\tau_n \geq 0$ such that $|\Phi^{\tau_n}_H(z_n)|\geq n$.

Moreover, the Birkhoff normal form of $H$ at the origin is divergent.

\end{theo} 

\medskip 

 We can modify the definition of the Hamiltonian on $\R^8$ as follows
\begin{align*} \tH_\om(x,y)&= (\om_1+I_4) I_1+(\om_2+I_4^2)I_2+(\om_3+I_4^3)I_3+\om_4 I_4,\\
\tH(x,y)&=\tH_\om(x,y)+\sum_{n\in \N} e^{-n(k_n+l_n)} I_3 F_n(x_1,x_2,y_1,y_2),\end{align*}
where $(k_n,l_n)\in \N^*\times \N^*$ are chosen so that \eqref{rational}and \eqref{lower} hold with $\om_2$ replaced by $(\om_2+I_{4,n}^2)$.

Here also, it is clear that $\nabla \tH_\om(I)$ is non-degenerate. We have the following.
\bigskip 

\begin{theo} -- \ \label{theo.44} The origin is a Lyapunov unstable equilibrium of the Hamiltonian flow $\Phi^t_{\tH}$ of $\tH$. Moreover, the Birkhoff normal form of $\tH$ at the origin is non-degenerate, hence the equilibrium is KAM-stable.
%Also, the Birkhoff normal form of $H$ at the origin is divergent.
  \end{theo}

\bigskip

\bluegray\subsection{The case of $\R^d \times \T^d$} \label{sec.torus}\color{black}
We see in this section how the constructions of Theorems \ref{theo.3} to \ref{theo.44} can be carried to the case of Hamiltonians on $\R^d \times \T^d$, $\T^d=\R^d/\Z^d$. We will only discuss the case of Theorem \ref{theo.4}, the others being similar.  Closely related to Hamiltonians as in \eqref{star} are the Hamiltonians on $\R^d \times \T^d$, $d\geq 4$, expressed in action-angle variables by \begin{align} \label{starr} \tag{$**$} H(r,\theta)&=H_\om(r)+{\mathcal O}_\theta^2(r),\\  \nonumber H_\om(r)&=(\omega_1+r_4)r_1+\sum_{j=2}^4 \om_j r_j, \end{align}
% H_\om(r)&=\begin{cases} \sum_{j=1}^3 \om_j r_j \text { if } d=3 \\ (\omega_1+r_4)r_1+\sum_{j=2}^4 \om_j r_j \text { if } d\geq4, \end{cases}
where ${\mathcal O}^2_\theta(r)$ denotes a real analytic Hamiltonian on $\R^d \times \T^d$ that is of order $2$ in the $r$ coordinates. For these Hamiltonians, the torus $\{0\}\times \T^d$ is invariant by the flow of $X_H=(\partial_\theta H,-\partial_r H)$, and the restricted dynamics  on this torus is a translation flow of frequency $\om$.

With  $\omega$ non-resonant, we again take  $(k_n,l_n) \in \Z^2$ such that  \begin{equation*} |k_n \om_1+ l_n\om_2| <\frac{1}{k_n^2}. \end{equation*}
Since we do not assume that $\om_1 \om_2<0$ it is possible that $k_nl_n<0$. Similarly to $F_n$ of \eqref{def.Fn}, we introduce 
$$\bar{F}_n(r_1,r_2,\theta_1,\theta_2)=r_1^{|k_n|}r_2^{|l_n|}\cos(2\pi(k_n \theta_1+l_n \theta_2))$$
and the real entire Hamiltonians 
\begin{align*} H(r,\theta)&=H_\om(r)+\sum_{n\in \N} e^{-n(k_n+l_n)} r_3 \bar{F}_n(r_1,r_2,\theta_1,\theta_2)\end{align*}
that satisfy \eqref{starr} and for which one can check similar results as those proved in Theorems \ref{theo.3} and \ref{theo.4}. 

Observe that in this action-angle setting, the fact that $k_n$ or $l_n$ may be negative does not constitute any restriction to the construction, and this is the reason why the condition $\om_1\om_2<0$ is not needed.

\bluegray
\section{Lyapunov unstable resonant equilibria on $\R^4$}  
\black

In case $\omega$ is resonant, it is known that instabilities are more likely to happen. Algebraic examples were known since long time ago \cite{LC,Ch} (see  \cite[\S 31]{MS}). Our construction is actually based on the existence in two degrees of freedom, for 
resonant frequencies, of polynomial  Hamiltonians that have  invariant lines that go through the origin such that any point on such a line converges to the origin for negative times and goes to infinity in finite time in the future.
 
Recall indeed the definition of the following real Hamiltonians for $k,l \in \N^* \times \N^*, k+l>2$
%\begin{equation*}F_n(x_1,x_2,y_1,y_2)=\xi_1^{k_n}\xi_2^{l_n} +\bxi_1^{k_n}\bxi_2^{l_n}.\end{equation*}
\begin{equation*}F_{k,l}(x_1,x_2,y_1,y_2)=\xi_1^{k}\xi_2^{l} +\eta_1^{k}\eta_2^{l}.\end{equation*}
We have 
\begin{prop}  -- \  \label{prop.diffuse} For any $n \in \N^*$, there exist $t_n \in [0,(2n)^{k+l-2}]$ such that $\Phi_{F_{k,l}}^{t_n}(B_{\frac{1}{2n}})\bigcap {B}_{2n}^c \neq \emptyset$. %and $z_n \in  \R^4$ such that $|z_n|=\frac{1}{n}$ and  $|\Phi_{F_n}^{t_n}(z_n)|= n+2$ and $\Phi_{F_n}^{t}(B_{\frac{1}{n}})\subset B_{n+2}$ for every $t\leq t_n$.
\end{prop}

If $\om_1$ and $\om_2$ are such that $k\om_1+l\om_2=0$, then the Hamiltonian flow of  $\xi_1^{k}\xi_2^{l} +\eta_1^{k}\eta_2^{l}$ commutes with that of 
$\om_1 I_1+\om_2 I_2$. Hence we get the following consequence:

 \begin{coro}  -- \ \label{rem.resonant} If $\om_1$ and $\om_2$ are such that $k\om_1+l\om_2=0$ for some $k,l\geq 1$ and $k+l>2$, then for any $a \in \R^*$, the flow of $H(x_1,x_2,y_1,y_2)=\om_1 I_1+\om_2 I_2 + a(\xi_1^{k}\xi_2^{l} +\eta_1^{k}\eta_2^{l})$ has an elliptic fixed point with frequency $(\om_1,\om_2)$ that is Lyapunov unstable. %\ref{sec.Fn}.
\end{coro}

%%%%%%%%%%%%%%%%%%%%%%%IIIIIICI

\noindent {\bg \sc Proof of Proposition \ref{prop.diffuse}} %It suffices to show that there exists $t'_n<e^{n(k_n+l_n)}$ and $|z'_n|=\frac{1}{n}$ such that  $\Phi_{F_n}^{t'_n}(z'_n)\geq n+2$. Indeed, we then take $t_n$ to be the first $t$ such that $\Phi_{F_n}^{t}(B_{\frac{1}{n}})$ intersects the circle around $0$ of radius $n+2$ and $z_n$ such that $|z_n|=\frac{1}{n}$ be a point such that $\Phi_{F_n}^{t_n}(z_n)=n+1$. Clearly $t_n\leq t'_n$ and the pair $(z_n,t_n)$ satisfies the conditions of the proposition. 

 %In the rest of this proof, 
 We let $u=\sqrt{l/k}$, $\a=k+l-1$. We assume $\a\geq 2$.  WLOG, we suppose that $u\geq 1$. 

Pick and fix $\nu, \nu' \in (0,1)$ such that 
\begin{equation*}-\frac{1}{4} +k\nu +l\nu'=1.\end{equation*}
Define a subset of $\R^4$,
\begin{equation*}\Delta := \left\{(x_1,x_2,y_1,y_2) \in \R^4 : (\xi_1,\xi_2)=\left(r e^{i2\pi \nu},u r e^{i2\pi \nu'}\right), r \in \R  \right\}.\end{equation*}
The Hamiltonian equations of $H$ give 
\begin{equation*}   \dot{\xi}_1 = -i k \eta_1^{k-1} \eta_2^l, \quad  \dot{\xi}_2 = -i l \eta_1^{k} \eta_2^{l-1}\end{equation*} 

For $(x_1,x_2,y_1,y_2) \in \Delta$, we get 
\begin{equation*}    \dot{\xi}_1 =  ku^l r^\a e^{i2\pi \nu}, \quad  \dot{\xi}_2 = ku^{l+1}r^\a e^{i2\pi \nu'} \end{equation*}%, % \\ \dot{\bxi}_1 = i k \xi_1^{k-1} \xi_2^l, \quad \quad  \dot{\bxi}_2 = i l \xi_1^{k} \xi_2^{l-1}  \end{equation*} 
which shows that $\Delta$ is invariant by the flow $\Phi_{F_n}^{t}$, and  moreover, that  the restriction of the vector field on $\Delta$ is given by 

\begin{equation*}\dot{r}=k u^{l}r^{\a}.\end{equation*}
Hence if we start with $r_0=\frac{1}{2n}$ we see that 
\begin{equation*}r(t)^{\a-1}=\frac{1}{(2n)^{\a-1}-(\a-1)ku^l t}.\end{equation*}
Define then $t_n$ such that $r(t_n)=2n+1$. Note that  $0\leq t_n \leq T_n:=(2n)^{\a-1}/(ku^l(\a-1))<(2n)^{\a-1}$ since $T_n$ is an explosion time of $r(t)$ with the initial condition $r_0=\frac{1}{2n}$. \bg 
$\Box$ 
\black

\bluegray
\section{Proofs of Theorems \ref{theo.3}--\ref{theo.44}} \label{sec555} \black % \ref{theo.3} and   \ref{theo.33}}
\black 
\bluegray

\subsection{Description of the diffusion mechanism} \label{sec.description}% \ref{theo.3} and   \ref{theo.33}}
 \color{black}
 
We first describe the proof of Theorem \ref{theo.3}, that is, diffusion in $3$ degrees of freedom near a close to resonant elliptic equilibrium.

{We want to exhibit diffusive orbits for the flow of \newline \indent $H_\om(x,y)+\sum_{n\in \N} e^{-n(k_n+l_n)} I_3 F_n(x_1,x_2,y_1,y_2)$ 
\begin{itemize}
\item From Corollary \ref{rem.resonant}, we know that if $k_n\bar \om_1+l_n\bar \om_2=0$ then the flow of $\bar \om_1 I_1+\bar \om_2 I_2+F_n(x_1,x_2,y_1,y_2)$ is unstable
\item Due to $(\mathcal L)$, an approximation argument (section \ref{sec.gronwall}) will show that, for fixed $I_3={\bf I}:=e^{-e^{n^3(k_{n}+l_{n})}}$, the flow of $H_\om(x,y)+ e^{-n(k_n+l_n)} {\bf I} F_n(x_1,x_2,y_1,y_2)$, has a point satisfying $I_1, I_2 \sim 1/n, I_3 ={\bf I}$, that escapes after a time much smaller that $ {\bf I}^{-1.1}$. 
\item The terms $F_l, l>n$ are too small and do not disrupt the diffusion at this time scale %$t=\mathcal O({\bf I}^{-1})$ 
 \item The terms $F_l, l<n$ average out to an $I_3^2$ term that contributes with $\mathcal O({\bf I}^2)$ magnitude at this level of $I_3$ and do not disrupt the diffusion at this time scale that is much smaller $ {\bf I}^{-1.1}$.\end{itemize}

In four degrees of freedom, we replace the almost resonance condition on $\om$ by the use of the fourth action variable that is also invariant along the flow. Indeed, when we fix the value of this variable to $I_{4,n}$ such that $k_n (\om_1+I_{4,n})+l_n \om_2=0$, we find our Hamiltonian exactly in the form to which we can apply the Corollary \ref{rem.resonant}. The variable $I_3$ plays then the same role as in the precedent case, of isolating the effect of a single $F_n$ in the diffusion, for various values of $I_3 \to 0$. 

With a bit more scrutiny in the terms of order $I_3^2$ of the Birkhoff normal form, we can easily see that these forms diverge in  our examples.

\bluegray
\subsection{Approximation by resonant systems and diffusive orbits} \label{sec.gronwall}% \ref{theo.3} and   \ref{theo.33}}
\black

\begin{lemm} -- \  \label{gronwall} There exists a constant $C_d>0$ such that the following holds. 
 Suppose $F,G \in C^2(\R^{2d},\R)$, $\om \in \R^d$,  $A,r,R, a, T>0$  such that $r\leq R$, $a^2 A T e^{d a A T} \leq 1/4$, and 
\begin{itemize} 

\medskip 

\item $\displaystyle{H(x,y)=\sum_{j=1}^d \om_j I_j +a F(x,y)}$

\medskip 

\item $\displaystyle{\bbH(x,y)=\sum_{j=1}^d \om_j I_j +a F(x,y)+a^2 G(x,y) }$

\medskip 

\item $\displaystyle{\|F\|_{C^2(B_{R+1})} \leq A, \quad \|G\|_{C^1(B_{R+1})} \leq A }$

\medskip 

\item $\displaystyle{\text{For \ all } s \in [0,T] : \Phi_H^{s}(B_r) \subset B_R }$
\end{itemize} 
Then, for all  $s \in [0,T]$ and for all $z \in B_r$ :
\begin{equation} \label{eqgro} \left| \Phi_H^{s}(z)-\Phi_{\bbH}^{s}(z)\right| \leq C_d a^2 A T e^{C_d a A T}\end{equation}
\end{lemm}

\bg 
\begin{proof} 
\black
Let $(X(s),Y(s)):=\Phi_H^{s}(z)$ and $(x(s),y(s)):=\Phi_{\bbH}^{s}(z)$.  
 Define the matrices $U_j=
\left(
\begin{matrix}
0 & -\omega_j  \\
\omega_j & 0
\end{matrix}
\right),$ and introduce the variables 
$$\left(
\begin{matrix}
u_j(s)  \\
v_j(s)
\end{matrix}
\right)=e^{sU_j}\left(
\begin{matrix}
x_j(s)-X_j(s)  \\
y_j(s)-Y_j(s)
\end{matrix}
\right).$$
Let $\xi(s)=(u_1(s),v_1(s),\ldots,u_d(s),v_d(s))$. Since $e^{sU_j}$ is a Euclidean isometry matrix,  \eqref{eqgro} is equivalent to proving
\begin{equation*} |\xi(s)| \leq  a^2 A T e^{d a A T}\end{equation*}

The Hamiltonian equations give that 
$$\left(
\begin{matrix}
\dot{u}_j(s)  \\
\dot{v}_j(s)
\end{matrix}
\right)=e^{sU_j}\left(
\begin{matrix}
a d_{x_j}F(\Phi_{\bbH}^{s}(z))-ad_{x_j}F(\Phi_{H}^{s}(z))+a^2d_{x_j}G(\Phi_{\bbH}^{s}(z))\\
a d_{y_j}F(\Phi_{\bbH}^{s}(z))-ad_{y_j}F(\Phi_{H}^{s}(z))+a^2d_{y_j}G(\Phi_{\bbH}^{s}(z))
\end{matrix}
\right).$$
Since, as long as $\Phi_H^{s}(z)$, and $\Phi_{\bbH}^{s}(z)$ are in $B_{R+1}$, we have that 
\begin{align*} |d_{x_j}F(\Phi_{\bbH}^{s}(z))-d_{x_j}F(\Phi_{H}^{s}(z))|&\leq \|F\|_{C^2(B_{R+1})} \sum \left(|x_j(s)-X_j(s)|+|y_j(s)-Y_j(s)|\right)\\
&\leq \sqrt{2d}  \|F\|_{C^2(B_{R+1})}  |\xi(s)|, \end{align*}
and a similar bound for the $y_j$ derivatives, the bounds on $F$ and $G$ then yield
\begin{equation*} |\dot{\xi}(s)|\leq 2d aA|\xi(s)|+\sqrt{2d} a^2A, \quad \xi(0)=0. \end{equation*}
Gronwall's inequality then implies that for some constant $C_d>0$, and as long as $\Phi_H^{s}(z)$, and $\Phi_{\bbH}^{s}(z)$ are in $B_{R+1}$ we have 
\begin{equation*} |\xi(s)|\leq C_d a^2Ase^{C_d aAs}. \end{equation*} 
Finally the condition $C_da^2 A T e^{C_d a A T} \leq 1/4$ allows to conclude, since it also  makes sure that  $\Phi_{\bbH}^{s}(B_r) \subset B_{R+1}$ for $s\in [0,T]$, from the fact that $\Phi_H^{s}(B_r) \subset B_R$.  \bg \end{proof} \black 
\bigskip 
\indent \quad \begin{coro} \black-- \  \label{cor.diff}
Let $a \in (e^{-2e^{n^3(k_n+l_n)}},e^{-e^{n^3(k_n+l_n)}})$. 
Let $H \in C^2(\R^{2d},\R)$ be such that 
\begin{equation*}H(x,y)= H_\om(x,y)+a F_n(x_1,x_2,y_1,y_2)+a^2 G_n(x,y)\end{equation*}
 with $\|G_n\|_{C^1(B_{2n})} \leq e^{4n(k_n+l_n)}$. 
 
 If \eqref{liouville} holds, there exist $t_n  \in [0,{(2n)^{k_n+l_n-2}}]$ and $z_n \in  \R^{2d}$ such that $|z_n|=\frac{1}{2n}$ and  $|\Phi_{H}^{\frac{t_n}{a}}(z_n)|\geq n$.% and $\Phi_{H}^{\frac{t}{a}}(B_{\frac{1}{n}})\subset B_{n+2}$ for every $t\leq t_n$.
\end{coro} 

\begin{proof} \black From \eqref{liouville}, there exists $\om'_1$ such that $|\om_1'-\om_1|< e^{-e^{n^4	(k_n+l_n)}}$ and $|k_n \om'_1+l_n\om_2|=0$. Then, $\{\om'_1 \xi_1\eta_1+ \om_2 \xi_2 \eta_2, F_n\}=0$. Hence if we define $\om'=(\om_1',\om_2,\ldots,\om_d)$ and 
\begin{equation*}H'(x,y)= H_{\om'}(x,y)+a F_n(x_1,x_2,y_1,y_2),\end{equation*} we get that 
\begin{equation*}|\Phi_{H'}^{\frac{t}{a}}(z)|=\left|\Phi_{\om'_1 I_1+ \om_2 I_2}^{\frac{t}{a}}\left(\Phi_{aF_n}^{\frac{t}{a}}(z)\right)\right|=|\Phi_{aF_n}^{\frac{t}{a}}(z)|=|\Phi_{F_n}^{{t}}(z)|.\end{equation*}
 Hence, by Proposition \ref{prop.diffuse}, there exists  $t_n \in[0,{(2n)^{k_n+l_n-2}}]$ and $z_n \in \R^{2d}$ such that $|z_n|\leq \frac{1}{2n}$, $|\Phi_{H'}^{\frac{t_n}{a}}(z_n)|=n+1$ and $\Phi_{H'}^{\frac{s}{a}}(B_{\frac{1}{n}})\subset B_{n+1}$ for every $s\leq t_n$.

Now since $|\om_1'-\om_1|< e^{-e^{n^4	(k_n+l_n)}}\leq a^2$, we have that 
\begin{equation*}H(x,y)= H_{\om'}(x,y)+aF_n(x_1,x_2,y_1,y_2)+a^2 G'_n(x,y)\end{equation*}
with  $\|G'_n\|_{C^1(B_{2n})} \leq e^{4n(k_n+l_n)}+1$. Note also that $\|F_n\|_{C^2(B_{2n})}\leq e^{n(k_n+l_n)}$.

Let $A=e^{4n(k_n+l_n)}+1$. Observe that for $T=\frac{t_n}{a}$, and $C_d$ as in Lemma \ref{gronwall}, we have  that
 $C_da^2 A T e^{C_d a A T} =C_da A t_n e^{C_d A t_n} \leq \frac1 4$.  We can thus apply Lemma \ref{gronwall}, with $r=\frac{1}{2n}$, $R=n+1$,  and deduce  that for all  $s \in [0,t_n]$ and for all $z \in B_{\frac{1}{2n}}$ :
 
\begin{equation*} \left| \Phi_H^{\frac{s}{a}}(z)-\Phi_{H'}^{\frac{s}{a}}(z) \right|\leq  a A t_n e^{d A t_n}\leq \frac1 4\end{equation*}
and the conclusion of the corollary thus holds if we apply the latter inequality to $z=z_n$ and $s=t_n$. \bg
\bg \end{proof} \black  
\black

\bluegray
\subsection{\sc Proof of Theorem \ref{theo.3}} \label{sec.theo3}
\black

We fix $n \in \N$ large, and want to show that there exists $z_n \in \R^6$, such that $|z_n|\leq \frac{1}{n}$, and $\tau_n \geq 0$ such that $|\Phi^{\tau_n}_H(z_n)|\geq n$. 

Note that for any value ${\bf I}\in \R_+$, the set $\{(x,y)\in \R^6 :  I_3=\xi_3\eta_3={\bf I}\}$ is invariant under all the flows we consider in this construction.

We restrict from here on our attention to 

\begin{equation}\label{eq.I3} I_3 = {\bf I}:=e^{-e^{n^3(k_{n}+l_{n})}}.\end{equation} 
For $r>0$, we  denote 
$$\hat{B}(r):=\{(x_1,x_2,y_1,y_2,x_3,y_3) : (x_1,x_2,y_1,y_2)\in B(r), I_3= {\bf I}\}.$$
In all this section, the norms $\| \cdot \|_{C^k(\hat{B}_{r})}$ will refer to the $C^k$ norm with respect to the variables $1$ and $2$ and not $3$. 

Introduce 
\begin{equation} \label{eqaj} a_j=e^{-{j}(k_{j}+l_{j})}, \quad  b_{j}=\frac{a_j}{k_{j} \om_1 +l_{j} \om_2}. \end{equation} 
Since $k_{n}\geq  e^{\frac{1}{|k_{j} \om_1+l_{j}\om_2|}}$ for any $j\leq n-1$, we have for  sufficiently large $n$
\begin{equation} \label{eqb} b_{j}\leq a_j\ln k_n.\end{equation} 

 Define the Hamiltonians on $\R^6$
\begin{equation*}\chi_{j}=-i b_{j} I_3  E_j, \quad  E_j= \xi_1^{k_{j}} \xi_2^{l_{j}}-\eta_1^{k_{j}} \eta_2^{l_{j}}.\end{equation*} 
that satisfy
\begin{equation} \label{echi} \{H_\om,\chi_{j}\}=-a_j I_3 F_{j}.\end{equation}

Let $\widehat \chi_{n-1}=\sum_{j\leq n-1} \chi_j$. 
It follows from \eqref{eqb} that for any fixed $k \in \N$, for sufficiently large $n$ 
\begin{equation} \label{echi_bound} \|\widehat \chi_{n-1}\|_{C^k(\hat{B}_{2n})}\leq  {\bf I} n \ln k_n\leq {\bf I}^{0.9}, \end{equation}
from the definition of ${\bf I}$ in \eqref{eq.I3}. We thus get for  $z\in \hat{B}(2n)$ %and provided $(x_3,y_3)$ are such that $\xi_3 \bxi_3={\bf I}$ 
%and because ${\bf I}= e^{-e^{n^3(k_{n}+l_{n})}}$
 \begin{equation} \label{eq.chi}\left| \Phi^1_{\widehat \chi_{n-1}}(z) - z \right| \leq {\bf I}^{0.8}.\end{equation} 
Next, we conjugate the flow of $H$ with the time one map of $\widehat \chi_{n-1}$. 
\footnote{In the computations that will follow, it is helpful to keep in mind that $a_jF_j$ and $a_j E_j$ are bounded in analytic norm, and that even if $b_j$ is large, it remains negligible compared to ${\bf I}^{-0.1}$.}

We have  \begin{equation*} \label{eqh} H \circ \Phi^1_{\widehat \chi_{n-1}} = H+\{H,\widehat \chi_{n-1}\}+\frac{1}{2!}\{\{H,\widehat \chi_{n-1}\},\widehat \chi_{n-1}\}+\ldots  \end{equation*} 
% with $$\mathcal R = \frac{1}{3!}\{\{\{H,\widehat \chi_{n-1}\},\widehat \chi_{n-1}\},\widehat \chi_{n-1}\}+\ldots$$
 Due to \eqref{echi} we have 
$$\{H,\widehat \chi_{n-1}\}=  -\sum_{j \leq n-1}  \III a_j F_j-  \III^2 \{ \sum_{j \geq 1} a_j  F_j,  \sum_{j \leq n-1}  i b_{j}  E_j\}$$
Hence 
\begin{align} \label{eqh2} H \circ \Phi^1_{\widehat \chi_{n-1}}&= H_\om+ \III a_n F_n+ \sum_{j \geq n+1}  \III a_j F_j+B_n  \III^2 + \mathcal R \\ \nonumber B_n&=-\frac{1}{2}  \{ \sum_{j \leq n-1} a_j  F_j,  \sum_{j \leq n-1}  i b_{j}  E_j\}- \{ \sum_{j \geq n} a_j  F_j,  \sum_{j \leq n-1}  i b_{j}  E_j\}\end{align}
where 
$$\mathcal R = \frac{1}{2}\{\{H-H_\om,\widehat \chi_{n-1}\},\widehat \chi_{n-1}\}+\frac{1}{3!}\{\{\{H,\widehat \chi_{n-1}\},\widehat \chi_{n-1}\},\widehat \chi_{n-1}\}+\ldots$$
Recall that $H(x,y)=H_\om(x,y)+\sum_{n\in \N} e^{-n(k_n+l_n)} I_3 F_n(x_1,x_2,y_1,y_2)$, hence due to \eqref{eq.I3} and  \eqref{echi_bound},  we have that 
$\mathcal R$ is a real analytic Hamiltonian that is of order $3$ in $I_3$ and satisfies 
\begin{equation} \label{R_bound}   \|\mathcal R\|_{C^1(\hB(2n))}\leq {\bf I}^{\frac{5}{2}}. \end{equation}
Now, since the $a_jF_j$ and $a_j E_j$ are bounded and converge exponentially fast to $0$ in analytic norm as $j \to \infty$, and by the bound \eqref{eqb} on the $b_j$ we have that  
\begin{equation} \label{two_bound} \|B_n\|_{C^1(\hB(2n))} \leq n \ln k_n.\end{equation} 
Since $k_{n+1} \geq {\bf I}^{-1}$ we have from the definition of $a_j=e^{-{j}(k_{j}+l_{j})}$ that 
\begin{equation} \label{one_bound} \| \sum_{j \geq n+1}  a_j F_j \|_{C^1(\hB(2n))} \leq {\bf I}\end{equation}

 Since $I_3$ is invariant by the Hamiltonian flow of all the functions we are considering, we now fix $I_3={\bf I}$ and consider the flow of $H \circ \Phi^1_{\widehat \chi_{n-1}}$ in restriction to the $(x_1,x_2,y_1,y_2)$ variables. 
Introduce $a:={\bf I}a_n$. We then have from  \eqref{eqh2},   \eqref{R_bound}, \eqref{two_bound} and \eqref{one_bound},  
\begin{equation*}\label{fconj}H \circ \Phi^1_{\widehat \chi_{n-1}} = H_\om+a F_{n}(x_1,x_2,y_1,y_2)+ a^2 G(x,y) \end{equation*}
with $\|G\|_{C^1(\hB_{2n})} \leq a_n^{-2}(1+n \ln k_n+{\bf I}^{1/2}) \leq e^{3n(k_{n}+l_{n})}.$

Observe that from the definitions of ${\bf I}$ and $a_n$ in \eqref{eq.I3} and \eqref{eqaj}, we have that $a={\bf I}a_n \in (e^{-2e^{n^3(k_n+l_n)}},e^{-e^{n^3(k_n+l_n)}})$.  Since \eqref{liouville} holds by hypothesis, we can thus apply Corollary \ref{cor.diff} and get that there exist $t_{n}  <(2n)^{k_{n}+l_{n}-2}$ and $w_{n} \in  \R^4$ such that $|w_{n}|=\frac{1}{2n}$ and  $|\Phi_{H\circ \Phi^1_{\widehat \chi_{n-1}}}^{\frac{t_{n}}{a}}(w_{n})|\geq n$. Now,  for $z_n$ we pick any $(x_3,y_3) \in \R^2$ such that $I_3={\bf I}$ and let $z_n=\Phi^{-1}_{\widehat \chi_{n-1}}(w_n,x_3,y_3)$. Thus, $|\Phi_{H}^{\frac{t_{n}}{a}}(z_{n})|\geq n$, while \eqref{eq.chi} implies that  $|z_n|\leq \frac{1}{n}$. This finishes the proof of Lyapunov instability. 

\bigskip

We now prove the divergence of the BNF of $H$. We first separate in $B_n$ of \eqref{eqh2} the monomials that only depend on the action variables from the other  non-resonant monomials. Recall that $\omega$ is assumed to be non-resonant, and that the non-resonant monomials are thus of the form  $c_{u_1,u_2,v_1,v_2}\xi_1^{u_1}\xi_2^{u_2}\eta_1^{v_1}\eta_2^{v_2}$ with $u_1\neq v_1$ or  $u_2 \neq v_2$. Observe that for $i,j \in \N$ we have that 
\begin{multline}\label{eqFE}
\{F_i,E_j\}=-ik_ik_j(\xi_1^{k_i-1}\xi_2^{l_i}\eta_1^{k_j-1}\eta_2^{l_j}+\eta_1^{k_i-1}\eta_2^{l_j}\xi_1^{k_j-1}\xi_2^{l_i})\\
-il_il_j(\xi_1^{k_i}\xi_2^{l_i-1}\eta_1^{k_j}\eta_2^{l_j-1}+\eta_1^{k_i}\eta_2^{l_i-1}\xi_1^{k_j}\xi_2^{l_j-1}).
 \end{multline}
Hence, the resonant monomials only come from $i=j$ and we have that 

\begin{equation} \label{h41} H \circ \Phi^1_{\widehat \chi_{n-1}} = H_\om-\sum_{j \leq n-1}\frac{a_j^2}{k_{j} \om_1 +l_{j} \om_2} I_3^2 \left(k_j^2 I_1^{k_j-1}I_2^{l_j}+l_j^2 I_1^{k_j}I_2^{l_j-1}\right)+
I+II+III\end{equation}
with
\begin{align*}
I&=I_3 \sum_{j \geq n} a_j  F_j(\xi_1,\xi_2,\eta_1,\eta_2)\\
II&=I_3^2 {\mathcal N} \\
III&={\mathcal R}=I_3^3W(x,y) \end{align*}
where $W$ is a real analytic Hamiltonian, and $\mathcal N$ is an analytic Hamiltonian  that is a sum of non-resonant terms
$${\mathcal N} = \sum_{u_1\neq v_1 \text{or } u_2 \neq v_2} c_{u_1,u_2,v_1,v_2}\xi_1^{u_1}\xi_2^{u_2}\eta_1^{v_1}\eta_2^{v_2}.$$
 
It is easy to see now that $H_\omega-\sum \frac{a_j^2}{k_{j} \om_1 +l_{j} \om_2} I_3^2 \left(k_j^2 I_1^{k_j-1}I_2^{l_j}+l_j^2 I_1^{k_j}I_2^{l_j-1}\right)$  is the ${\mathcal O}(I_3^2)$ part of the BNF of $H$ at the origin.  Indeed, define  \begin{multline*} \mathcal A_n := \left\{ (u_1,u_2,v_1,v_2)\in \N^4 : {u_1\neq v_1 \text{or } u_2 \neq v_2} \right. \\ \left. \text{ and }  u_1+u_2<k_{n-1}+l_{n-1}, v_1+v_2<k_{n-1}+l_{n-1}\right\}\end{multline*}  and 
\begin{equation}\label{eq.psi} \psi= I_3^2 \sum_{\mathcal A_n} \frac{-ic_{u_1,u_2,v_1,v_2}}{(u_1-v_1)\om_1+(u_2-v_2)\om_2} \xi_1^{u_1}\xi_2^{u_2}\eta_1^{v_1}\eta_2^{v_2}\end{equation}
and observe that since 
\begin{equation*}  \{H_\om,\psi \}=-  I_3^2\sum_{\mathcal A_n} c_{u_1,u_2,v_1,v_2}\xi_1^{u_1}\xi_2^{u_2}\eta_1^{v_1}\eta_2^{v_2}     \end{equation*}
then \eqref{h41} gives 
\begin{align} \label{eq.BNF} H \circ \Phi^1_{\widehat \chi_{n-1}}\circ \Phi^1_\psi&= H_\om-\sum_{j \leq n-1}\frac{a_j^2}{k_{j} \om_1  +l_{j} \om_2} I_3^2 \left(k_j^2 I_1^{k_j-1}I_2^{l_j}+l_j^2 I_1^{k_j}I_2^{l_j-1}\right)\\ \nonumber &+
I'+II'+III'\end{align}
where $I',II',III'$ are real analytic Hamiltonians around the origin of the form
\begin{align*}
I'&=I=I_3 \sum_{j \geq n} a_j  F_j(\xi_1,\xi_2,\eta_1,\eta_2)\\
II'&=I_3^2\sum_{u_1+u_2\geq k_{n-1}+l_{n-1}  \text{or } v_1+v_2\geq k_{n-1}+l_{n-1}} c_{u_1,u_2,v_1,v_2}\xi_1^{u_1}\xi_2^{u_2}\eta_1^{v_1}\eta_2^{v_2}  \\
III'&=I_3^3W'(x,y) \end{align*}
where $W'$ is a real analytic Hamiltonian. Hence, the terms in $III'$ do not contribute to the $O^2(I_3)$ part of the BNF of $H$ at $0$.

 Since the order of the $(\xi_1,\xi_2,\eta_1,\eta_2)$-terms in $I'$ and $II'$ are  higher than $k_{n-1}+l_{n-1}$, and since \eqref{eq.BNF} holds for an arbitrary $n$, we conclude that  the $O^2(I_3)$ part of the BNF of $H$ is given by 
\begin{equation*}
\sum_{j=1}^\infty -\frac{e^{-2{j}(k_{j}+l_{j})}}{k_{j} \om_1 +l_{j} \om_2}   \left(k_j^2 I_1^{k_j-1}I_2^{l_j}+l_j^2 I_1^{k_j}I_2^{l_j-1}\right)I_3^2
\end{equation*}
which is clearly divergent since $|{k_{j} \om_1 +l_{j} \om_2}|<e^{-e^{j^4(k_j+l_j)}}$ by our Liouville hypothesis \eqref{liouville}. Observe that, in fact, the super Liouville condition $|{k_{j} \om_1 +l_{j} \om_2}|<e^{-j^2(k_j+l_j)}$ is sufficient for the divergence of the BNF. \bg $\hfill \Box$ \black

\bluegray
\subsection{\sc Proof of Theorem \ref{theo.33}} 
\black

We keep the same definitions of $\chi_j$ and $\widehat \chi_j$ as in the above proof and observe that, since since $H-\tH=I_3^3(I_1+I_3I_2)$, it still holds that for $a=a_n{\bf I}$, ${\bf I}=e^{-e^{n^3(k_{n}+l_{n})}}$, we have  
 \begin{equation*}\tH \circ \Phi^1_{\widehat \chi_{n-1}} =
 H_\om+a F_{n}(x_1,x_2,y_1,y_2)+ a^2 \tilde{G}(x,y) \end{equation*}
with $\|\tilde{G}\|_{C^1(\hB_{2n})} \leq  e^{3n(k_{n}+l_{n})}.$
The rest of the proof of the topological instability of the origin is the same as that of Theorem \ref{theo.3}. \bg $\hfill \Box$ \black %Finally, observe that since  $H-\tH=I_3^3(I_1+I_3I_2)$, then the BNF of $H$ and $\tH$ at the origin coincide up to order $O^2(I_3)$, hence the divergence of the BNF of $\tH$ follows from that of $H$. 

%Pick any $z_3=(x_3,y_3)$ such that $I_3={\bf I}$ and define $\zeta_{n+1}=(z_{n+1},z_3)$ and observe that   for $\tau_{n+1}=\frac{t_{n+1}}{a}$, we have   

\bluegray
\subsection{\sc Proof of Theorem \ref{theo.4}}
\black 

 %We fix $n \in \N$ and want to show that there exists $z_n \in \R^8$, such that $|z_n|\leq \frac{1}{n}$, and $\tau_n \geq 0$ such that $|\Phi^{\tau_n}_H(z_n)|\geq n$. 

Note that for any value of $({\bf I}, {\bf J})\in \R_+\times \R_+^*$, the set  $\{(x,y)\in \R^8 :  I_3=\xi_3\eta_3={\bf I}, \xi_4\eta_4={\bf J} \}$ is invariant under all the flows we consider in this construction. 

If we fix now $I_4={\bf J}:=I_{4,n}$ and $I_3={\bf I}:=e^{-e^{n^3(k_{n}+l_{n})}}$,  the restriction of the flow of $H$ to the $(x_1,x_2,y_1,y_2)$ space takes the form:
\begin{equation*} \label{eq.h4} H(x,y)=  (\om_1+{\bf J})I_1+ \om_2 I_2+\om_3 {\bf I}+\om_4 {\bf J}+ \sum_{n\in \N} {\bf I} e^{-n(k_n+l_n)}  F_n(x_1,x_2,y_1,y_2)\end{equation*}
which has the same flow as in 
% We will finish if we show that there exists $z_n \in \R^4$, such that $|z_n| \leq \frac{1}{n}$, and $\tau_n \geq 0$ such that $|\Phi^{\tau_n}_H(z_n)|\geq n$. Indeed, for $z_n$ we then pick any $(\xi_3,\eta_3)$ real such that $\xi_3 \eta_3={\bf I}$, any $(\xi_4,\eta_4)$ real such that $\xi_4 \eta_4={\bf J}$ and let $z_n=(z_n,\xi_3,\eta_3,\xi_4,\eta_4)$. 
%This flow has exactly the same form as 
%\eqref{eq.h3} in 
the proof of Theorem \ref{theo.3} with this difference that 
 $\omega_1$ is replaced by $\omega_1+{\bf J}$. Moreover, the hypothesis \eqref{rational} and \eqref{lower} of Theorem \ref{theo.4} imply hypotheses ($\mathcal L$) and \eqref{lower} of Theorem \ref{theo.3}, so the existence of the diffusive orbit follows from  Theorem \ref{theo.3}.

As for the BNF, using hypothesis \eqref{lower} we can define  $\widehat \chi_{n-1}$ and $\psi$ as in \eqref{eq.chi} and \eqref{eq.psi}, but with $\om_1+I_4$ instead of $\om_1$, and get, in sufficiently small neighborhood of the origin:
\begin{align*} \label{eq.BNF2} H \circ \Phi^1_{\widehat \chi_{n-1}}\circ \Phi^1_\psi&= H_\om-\sum_{j \leq n-1}\frac{a_j^2}{k_{j} (\om_1+I_4)  +l_{j} \om_2} I_3^2 \left(k_j^2 I_1^{k_j-1}I_2^{l_j}+l_j^2 I_1^{k_j}I_2^{l_j-1}\right)\\ \nonumber &+
I'+II'+III'\end{align*}
where $I',II',III'$ are real analytic Hamiltonians around the origin  (for this, we restrict to $I_4\ll 1$) and are of the same form as in Section \ref{sec.theo3}.
Hence, the terms in $III'$ do not contribute to the $O^2(I_3)$ part of the BNF of $H$ at $0$. As in the proof of Theorem \ref{theo.3}, we conclude that  the $I_3^2   I_1^{k_{n-1}-1}I_2^{l_{n-1}}$ term of the BNF of $H$ is given by 
\begin{equation*}
 -\frac{e^{-2(n-1)(k_{n-1}+l_{n-1})}}{k_{n-1} (\om_1+I_4) +l_{n-1} \om_2}   k_{n-1}^2  I_3^2   I_1^{k_{n-1}-1}I_2^{l_{n-1}}
\end{equation*}
which, as a series in $I_4$ has a radius of convergence smaller than \newline $|{k_{n-1}\om_1+l_{n-1} \om_2}|/k_{n-1}$, because we assumed that $k_{n-1} \om_1- l_{n-1}\om_2<0$.  Since this holds for an arbitrary $n$, and since  $|{k_{n-1}\om_1+l_{n-1} \om_2}|/k_{n-1}$ goes to $0$ as $n \to \infty$, we conclude that the BNF of $H$ diverges. \bg $\hfill \Box$ \black

%Pick any $z_3=(x_3,y_3)$ such that $I_3={\bf I}$ and define $z_{n+1}=(z_{n+1},z_3)$ and observe that   for $\tau_{n+1}=\frac{t_{n+1}}{a}$, we have   

\bluegray
\subsection{\sc Proof of Theorem \ref{theo.44}}
\black
We just replace $\om_2$ by $\om_2+{\bf J}^2$ everywhere in the proof of the topological instability of the fixed point in Theorem \ref{theo.4}. \bg $\hfill \Box$ \black 

\medskip

\bluegray
\section{Divergent Birkhoff normal forms with arbitrary frequencies in all degrees of freedom larger than $1$} \label{sec.newBNF}
\black

In this section, we prove Theorem {\bg B} by giving explicit examples 
of real entire Hamiltonians having an elliptic equilibrium at the origin and a divergent BNF for {\it any} degree of freedom including $d=2$, and for any non-resonant frequency vector, including vectors whose coordinates are all of the same sign. Of course we can treat just the case $d=2$ since for the case $d\geq 3$ it is sufficient to add to the Hamiltonians defined in $d=2$ a trivial integrable part $\sum_{j=3}^d \om_j I_j$.

Because it is of independent interest, this part is written in a self contained way and can be read independently from sections \ref{sec.statements} to \ref{sec555}. The cost is that some notations and definitions are reintroduced, which causes slight redundancies with the former sections. 

We suppose $\omega_1\omega_2<0$ and will explain later what difference should be applied to treat the case 
$\omega_1\omega_2>0$.

WLOG, we can assume that $|\omega_2|=\theta |\omega_1|$ for some $\theta >1$  that will be fixed in all the sequel. 

We define a sequence $(k_n,l_n) \in \N^2$  such that 
\begin{equation} \label{Bk} |k_n \omega_1 +l_n \omega_2| <\frac{1}{k_n}, \quad k_n \geq 10e^{e^n}.\end{equation}

 The difference with the examples of Theorems \eqref{theo.3} to \eqref{theo.44} is that when $d=2$ we do not have the extra action variables $I_3$ and $I_4$ that were instrumental in the diffusion mechanism as well as in the proof of divergence of the BNF in these constructions. Instead, we intend to give to $I_2$ a double role that includes the role of $I_4$ in the proof of divergence of the BNF of Theorems \ref{theo.4} and \ref{theo.44}.  Introduce the integrable Hamiltonian 

\begin{equation*}H_\om(x,y)=(\omega_1+I_2)I_1+\omega_2I_2,\end{equation*}

In the construction we will use a sequence  of numbers  $\zeta_n\in [0,1]$ that we will fix inductively in the proof. For any choice of the sequence $(\zeta_n)_{n\in \N}\in [0,1]^\N$, we define a real entire Hamiltonian on $\R^4$ as follows 
  
\begin{equation} \label{Bham}H(x,y)=H_\om(x,y)+\sum_{n\in \N}\zeta_n  \tF_n(x_1,x_2,y_1,y_2),\end{equation}
where 
 \begin{equation*} \label{def.tFn} \tF_n(x_1,x_2,y_1,y_2)=a_n(\xi_1^{k_n}\xi_2^{l_n} +\eta_1^{k_n}\eta_2^{l_n}), \quad a_n=e^{-n(k_n+l_n)}. \end{equation*} 
\begin{theo} \label{Bthm1} For $\om$ and $(k_n,l_n)$ as in \eqref{Bk}, there exists $(\zeta_n)_{n\in \N}\in [0,1]^\N$ such that the Hamiltonian $H$ as in \eqref{Bham} has an elliptic equilibrium at the origin with a divergent Birkhoff Normal form. \end{theo}

Similarly to the proof 
of divergence of the BNF of Theorems \ref{theo.4} and \ref{theo.44}, the main ingredient is the appearance in  the computation of the BNF of terms that include a pole close to $I_2=0$. Unlike $I_4$, the term $I_2$ in  $(\omega_1+I_2)I_1$ is not decoupled from the terms $\tF_n$ that break the integrability of the Hamiltonian. For that reason, the computations of the BNF will be more involved, but the advantage compared to the proof of Theorem {\bg A} is that in the current section we just want to prove the divergence of the BNF without any further information on the dynamics near the origin of the Hamiltonian flow.

The key to the proof is an explicit formal computation of some coefficients of the BNF of H. First, fix $\eps\in (0,0.01)$ such that $1+2\eps<\theta$ and consider the integer 
\begin{equation} \label{Bkn}  \hat k_n=[(1+\eps)l_n] <k_n\end{equation}

\begin{prop} The coefficient $\Gamma_n$ of $I_1^{k_n-1}I_2^{\hk_n}$ in the BNF of $H$ at the origin is given by 
\begin{equation} \label{Bcoeff} 
\Gamma_n=\zeta_n^2 \gamma_n +\zeta_nP_n(\zeta_0,\ldots,\zeta_{n-1})+Q_n(\zeta_0,\ldots,\zeta_{n-1}),
\end{equation}
where $P_n$ and $Q_n$ are polynomials (that depend on $\om$ and $(k_j,l_j)$ for $j\leq n$), and 
\begin{equation}\label{Bvn} \gamma_n= (-1)^{\hk_n-l_n} a_n^2 k_n \left(\frac{k_n}{k_n\om_1+l_n \om_2} \right)^{\hk_n-l_n+1}. \end{equation} 
In particular,  \eqref{Bk} and \eqref{Bkn} imply that $|\gamma_n|\geq  e^{nk_n}$.
\label{Bkey} 
\end{prop}
\bg 
\begin{proof}[\sc Proof of Theorem \ref{Bthm1}] \black From Proposition \ref{Bkey}, it is clear that once $\zeta_0,\ldots,\zeta_{n-1}$ are fixed, we can choose $\zeta_n \in [0,1]$ such that $|\Gamma_n| \geq e^{\frac{n}{2}k_n}$. Doing so for all $n\in \N$,  implies that the radius of convergence of the BNF of H at the origin is zero and finishes the proof of Theorem \ref{Bthm1}.\footnote{Note that, whatever values are taken by $P_n(\zeta_0,\ldots,\zeta_{n-1})$ and $Q_n(\zeta_0,\ldots,\zeta_{n-1})$, only a very small measure of $\zeta_n \in [0,1]$ would give  $|\Gamma_n|\geq  e^{\frac{n}{2}k_n}$. Hence, the prevalent choice of the sequence $\{\zeta_n\}$ will give a divergent BNF. This however, follows from the existence of just one example by the dichotomy result of Perez-Marco.} \bg \end{proof} \black

\bg \begin{proof}[\sc Proof of Proposition \ref{Bkey}] \black  We first make the following elementary observations regarding our strategy for the computation of the BNF and in particular of the coefficient $\Gamma_n$ in front of $I_1^{k_n-1}I_2^{\hk_n}$.  
\begin{itemize}
\bg \item \black The terms $\sum_{m\geq n+1}\zeta_m \tF_m(x_1,x_2,y_1,y_2)$ are of much higher degree than $2k_n+2\hk_n$ and thus do not affect the coefficient $\Gamma_n$.
\bg \item \black When performing the BNF reduction steps of $\sum_{m\leq n}\zeta_m \tF_m(x_1,x_2,y_1,y_2)$, all the coefficients of the monomials that appear inductively are polynomials in the variables $(\xi_1,\ldots,\zeta_n)$, everything else being fixed. 
\bg \item \black Due to \eqref{Bkn}, the terms that appear with powers $d\geq 3$ of $\zeta_n$ have degree higher than $3(k_n+l_n)-4>2k_n+2\hk_n$, hence these terms do not contribute to the coefficient $\Gamma_n$.
\end{itemize}
In conclusion, the coefficient $\Gamma_n$ has the quadratic form in $\zeta_n$ as in \eqref{Bcoeff} 	and we just need to estimate the coefficient $\gamma_n$ that comes with $\zeta_n^2$.

Recall that $a_n=e^{-n(k_n+l_n)}$, and  define the following real Hamiltonian on $\R^4$
\begin{equation*}\chi_{n}=-i\zeta_{n}  \tE_n U_n :  \  \tE_n=a_n(\xi_1^{k_{n}} \xi_2^{l_{n}}-\eta_1^{k_{n}} \eta_2^{l_{n}}), \ U_n=\frac{1}{k_n(\om_1+I_2)+l_n\om_2}. \ \  \end{equation*} 
Note that $U_n$ has a pole as a function of $I_2$ near $0$. However, it is still a nice analytic function in a tiny  neighborhood of $0$. This is the case for all the Hamiltonians we will consider such as $\chi_n$ for example. We will not need this fact in reality, and all the calculations below can be viewed as operations on merely formal  power series. 

Using Leibniz's product rule for Poisson brackets, we have that $ \{ H_\omega,  E_n U_n\}=U_n   \{ H_\omega,  E_n \}$ since both $H_\omega$ and $U_n$ are integrable. But $ \{ H_\omega,  iE_n \}=(k_n\omega_1+l_n\omega_2)F_n +(k_n I_2 + l_n  I_1) F_n$, hence multiplying by $\zeta_nU_n$ and regrouping gives
 \begin{equation}\label{Hchi} \{H_\om,\chi_{n}\}=-\zeta_n\tF_{n}-\zeta_nl_nI_1\tF_nU_n.\end{equation} 
%and introduce $$\widehat \chi_{n-1}=\sum_{j\leq n-1} \chi_j.$$
This will be used to show that if we denote by $\Phi^1_{\chi_n}$  the time one map of the Hamiltonian flow of $\chi_n$, then we get the following.

\begin{lemm} \label{Breduction}We have that 

\begin{equation} \label{Bh41} H \circ \Phi^1_{\chi_{n}} = H_\om+\sum_{j\leq n-1}\zeta_j \tF_j-{\zeta_n^2} a_n^2 k_n^2   I_1^{k_n-1}I_2^{l_n}U_n+I_1^{k_n}h_n(I_2)+
\zeta_n\cR_1+\cR_2\end{equation}
with
\begin{itemize}
\bg \item \black $h_n(I_2)$ is a function of the action variable $I_2$;
\bg \item \black $\cR_2$ containing monomials of degree larger than $2k_n+2\hk_n$;% (with coefficients that may depend on $\zeta_n$). 
\bg \item \black $\cR_1$ containing non-resonant monomials with the sum of the degrees of $\xi_1$ and $\eta_1$ for each monomial being strictly larger than $k_n$.\footnote{Note that we allow the terms in $\cR_1$ to depend on $\zeta_n$, as well as the terms in $\cR_2$ and in $h_n(I_2)$.}
%with coefficients that depend only on $(\xi_1,\ldots,\zeta_{n-1})$
%\bg \item \black $I$ containing monomials with coefficients that depend only on $(\xi_1,\ldots,\zeta_{n-1})$, and 
\end{itemize}
\end{lemm}

 Before proving the lemma, we see how it implies Proposition \ref{Bkey}. 
 
 \medskip 
 
\bg  \noindent {\bf 1. }  \black First, observe that the terms in $\cR_2$ do not affect the coefficient $\Gamma_n$ because their degree is too large. 

 \medskip 
 
\bg  \noindent {\bf 2. }  \black Second, observe that the term $I_1^{k_n}h_n(I_2)$ does not affect $\Gamma_n$ because its degree in $I_1$ is higher than $k_n-1$. 

 \medskip 
 
\bg  \noindent {\bf 3.} \black Thirdly, the terms in $\sum_{j\leq n-1}\zeta_j \tF_j$ and $\zeta_n\cR_1$ will contribute to $\Gamma_n$ only in the part $\zeta_nP_n(\zeta_0,\ldots,\zeta_{n-1})+Q_n(\zeta_0,\ldots,\zeta_{n-1})$. Indeed, any term of degree $2$ in $\zeta_n$ that will come from further reduction of the non-resonant terms of $\sum_{j\leq n-1}\zeta_j \tF_j+\zeta_n\cR_1$  will have a degree higher than $k_n$ in $I_1$ (this is because the sum of the degrees of $\xi_1$ and $\eta_1$ for each monomial in $\cR_1$ is  larger or equal to $k_n+1$ and  $2(k_n+1)-2\geq 2k_n$). 

 \medskip 
 
\bg  \noindent {\bf 4. }  \black Since $U_n=\frac{1}{k(\om_1+I_2)+l\om_2}$, the expansion of $U_n$ in a power series in $I_2$ gives that the coefficient of $I_1^{k_n-1}I_2^{\hk_n}$ coming from $ a_n^2k_n^2 I_1^{k_n-1}I_2^{l_n}U_n$ is exactly $\gamma_n$ given by \eqref{Bvn}.
 
 \medskip 
 
\bg  \noindent {\bf 5. }  \black Finally, it is straightforward that  \eqref{Bk} and \eqref{Bkn} imply that $|\gamma_n|\geq  e^{nk_n}$.
  
  \medskip 
  
 We turn now to the proof of the lemma. 
  
\bg \begin{proof}[\sc Proof of Lemma  \ref{Breduction}] \black In all this proof $\cR_2$ will denote a sum of monomials of degree larger  than $2k_n+2\hk_n$ and may be different from line to line. By $\cR_1$ we denote sums of 
non-resonant monomials with the sum of the degrees of $\xi_1$ and $\eta_1$ for each monomial being strictly larger than $k_n$. Observe for example that the first terms that will appear in $\cR_1$ come from $\{F_j,\chi_n\}$ for $j\leq n-1$.

We have  \begin{align} \label{Beqh} H \circ \Phi^1_{\chi_{n}}&= H+\{H,\chi_{n}\}+\frac{1}{2!}\{\{H,\chi_{n}\},\chi_{n}\}+\ldots\\ \nonumber &= H+\{H,\chi_{n}\}+\frac{1}{2!}\{\{H,\chi_{n}\},\chi_{n}\}+\cR_2.
  \end{align}   Next, \eqref{Hchi} implies that 
\begin{align}\nonumber 
 \{H,\chi_{n}\}&=\{H_\om,\chi_n\}+\sum_{j\in \N}\zeta_j  \{\tF_j,\chi_n\} \\ \label{B101}
&=-\zeta_n\tF_n-\zeta_nl_n\tF_nI_1U_n-i\zeta_n^2\{\tF_n,\tE_nU_n\} + \sum_{j\neq n}\zeta_j\zeta_n\{\tF_j,\tE_nU_n\} \\ \label{B11} 
&=-\zeta_n\tF_n-i\zeta_n^2\{\tF_n,\tE_nU_n\}+\zeta_n\cR_1.
 %\\  \nonumber &=-+\zeta_n\tF_n-i\zeta_n^2\{\tF_n,\tE_nU_n\}+\zeta_n\cR_1'
\end{align} 
 Now, since the terms $\{\{\tF_j,\chi_n\},\chi_n\}$ for $j\neq n$ are non-resonant, \eqref{B101}  leads to 
\begin{equation} 
\label{B22}
\{\{H,\chi_{n}\},\chi_n\}=i\zeta_n^2\{\tF_n,\tE_nU_n\}+ i\zeta_n^2l_n \{\tF_nI_1U_n,\tE_nU_n\}+\zeta_n \cR_1+\cR_2.\end{equation}

%From \eqref{B11} we also see that 
%\begin{equation} 
%\label{B111} \{H,\chi_{n}\}=-\zeta_n\tF_n-i\zeta_n^2\{\tF_n,\tE_nU_n\}+\zeta_n\cR_1.
%\end{equation} 
Putting together \eqref{Beqh}, \eqref{B11} and \eqref{B22}, we see that 
\begin{multline} \label{BBB} H \circ \Phi^1_{\chi_{n}} = H_\om+\sum_{j\leq n-1}\zeta_j \tF_j-\frac{1}{2}i\zeta_n^2\{\tF_n,\tE_nU_n\} \\ +\frac{1}{2} i\zeta_n^2l_n \{\tF_nI_1U_n,\tE_nU_n\}+\zeta_n\cR_1+\cR_2\end{multline}

In the lines below the notation $h_n(I_2)$ refers to a power series in $I_2$ that may change from line to line. Observe that \eqref{eqFE} implies that 
\begin{align}
\label{B33} \{\tF_n,\tE_n\}U_n&=-2ia_n^2\left(k_n^2 I_1^{k_n-1}I_2^{l_n}+l_n^2 I_1^{k_n}I_2^{l_n-1}\right)U_n \\\nonumber &=-2ia_n^2 k_n^2 I_1^{k_n-1}I_2^{l_n}U_n+I_1^{k_n}h_n(I_2) \\ \label{B44} 
\{\tF_n,U_n\}\tE_n&=a_n^2l_nk_n\left(\xi_1^{2k_n}\xi_2^{2l_n}+\eta_1^{2k_n}\eta_2^{2l_n}-2I_1^{k_n}I_2^{l_n}\right)U_n^2 \\&=\nonumber I_1^{k_n}h_n(I_2)+\cR_1\\
\label{IF} \{I_1,\tE_n\}\tF_n&=-k_n\tF_n^2=I_1^{k_n}h_n(I_2)+\cR_1.
\end{align}
It follows from \eqref{B33} and \eqref{B44} that 
\begin{equation} \label{eq1000} -\frac{1}{2}i\zeta_n^2\{\tF_n,\tE_nU_n\}=-\zeta_n^2 a_n^2 k_n^2 I_1^{k_n-1}I_2^{l_n}U_n+I_1^{k_n}h_n(I_2)+\zeta_n\cR_1\end{equation} 
It follows from \eqref{B33}, \eqref{B44} and \eqref{IF} that 
\begin{equation} \label{eq2000} \frac{1}{2} i\zeta_n^2l_n \{\tF_nI_1U_n,\tE_nU_n\}=I_1^{k_n}h_n(I_2)+\zeta_n\cR_1.\end{equation} 
Finally, \eqref{BBB}, \eqref{eq1000} and \eqref{eq2000} imply the result of Lemma \ref{Breduction}. \bg \end{proof} \black 

The proof of Proposition \ref{Bkey} is thus complete.  \bg \end{proof} \black 

\bg 
\begin{proof}[\sc Proof of Theorem {\bg B}] \black To finish the proof of Theorem {\bg B}, we need to consider the case where $\omega_1\omega_2>0$. 

We then have  a sequence $(k_n,l_n) \in \N^2$  such that  
\begin{equation} \label{Bkkk} |k_n \omega_1 -l_n \omega_2| <\frac{1}{k_n}, \quad k_n \geq 10e^{e^n}.\end{equation}

Next, we replace the definition of the Hamiltonian in \eqref{Bham} by
\begin{equation} \label{Bham2}H(x,y)=H_\om(x,y)+\sum_{n\in \N}\zeta_na_n(\xi_1^{k_n}\eta_2^{l_n} +\eta_1^{k_n}\xi_2^{l_n}).\end{equation}
The proof that the BNF of the Hamiltonian in \eqref{Bham2} under the condition \eqref{Bkkk} follows then the same lines as that for the Hamitlonian \eqref{Bham} under the condition \eqref{Bk}.
 \bg \end{proof} \black

\bigskip 
\bigskip

\noindent {\sc \bluegray Acknowledgment.}  I am grateful to H. Eliasson and M. Saprykina  for their  inputs to this paper. I am also grateful to A. Bounemoura, A. Chenciner, and R. Krikorian for valuable discussions on the topic. Thank you Tolya for always pushing me to work on this problem.

\bigskip 

\bluegray
\Large

\end{document}